\documentclass[11pt]{article}
\usepackage{caption}
\usepackage{subfigure}
\usepackage{indentfirst}
\usepackage{latexsym}
\usepackage{hyperref}
\usepackage{graphicx}
\usepackage{mathrsfs}
\usepackage{bookmark}
\usepackage{enumerate}
\usepackage{amsfonts,amsthm,amssymb,mathtools}
\usepackage{amsmath}
\usepackage{latexsym, euscript, epic, eepic, color}
\usepackage{amssymb}

\usepackage{enumerate}
\usepackage[all]{xy}
\usepackage{setspace}
\allowdisplaybreaks

\usepackage{epstopdf}
\numberwithin{equation}{section}
\newtheorem{theorem}{Theorem}[section]
\newtheorem{corollary}[theorem]{Corollary}
\newtheorem{definition}[theorem]{Definition}

\newtheorem{lemma}[theorem]{Lemma}

 \allowdisplaybreaks

\begin{document}
\baselineskip=16pt

\title{The maximum $A_{\alpha}$-spectral radius of $t$-connected graphs with bounded matching number}

\author{\textbf{Chang Liu, Zimo Yan, Jianping Li\textsuperscript{\thanks{Corresponding author.}}}\\
	\small \emph{College of Liberal Arts and Sciences, National University of Defense Technology,}\\[-0.8ex]
	\small \emph{Changsha, Hunan, 410073, P. R. China}\\[-0.4ex]
		\small\tt clnudt19@126.com\\
		\small\tt yanzimo20@nudt.edu.cn\\
		\small\tt lijianping65@nudt.edu.cn\\
	}

\date{\today}

\maketitle

\begin{abstract}
	Let $G$ be a graph with adjacency matrix $A(G)$ and let $D(G)$ be a diagonal matrix of the degrees of $G$. In \cite{Niki1}, Nikiforov defined the $A_{\alpha}$-matrix of $G$ as
	\begin{equation*}
		A_{\alpha}(G)=\alpha D(G)+(1-\alpha)A(G),
	\end{equation*}
	where $\alpha\in[0,1]$ is an arbitrary real number. The largest eigenvalue of $A_{\alpha}(G)$ is called the $A_{\alpha}$-spectral radius of $G$. Let $n$, $t$, $k$ be positive integers, satisfying $t\geq1$, $k\geq2$, $n\geq k+2$, and $n\equiv k$ (mod $2$). In this paper, for $\alpha\in[0,\frac{1}{2}]$, we determine the extremal graphs with the maximum $A_{\alpha}$-spectral radius among all $t$-connected graphs on $n$ vertices with matching number $\frac{n-k}{2}$ at most. This generalizes some results of O (2021) \cite{O21} and Zhang (2022) \cite{Wzhang}.
	\\[2pt]
\end{abstract}

\section{Introduction}
All graphs considered here are simple, undirected and connected. Let $G$ be an $n$-vertex graph with vertex set $V(G)=\left\lbrace v_1,v_2,\cdots,v_n \right\rbrace $ and edge set $E(G)$. The adjacency matrix of $G$ is defined as $A(G)=(a_{ij})_{n\times n}$, where $a_{ij}=1$ if $v_iv_j\in E(G)$, and $a_{ij}=0$ otherwise. Let $D(G)$ be the diagonal matrix of the degrees of $G$. In \cite{Niki1}, Nikiforov defined the $A_{\alpha}$-matrix of $G$ as
\begin{equation*}
	A_{\alpha}(G)=\alpha D(G)+ (1-\alpha) A(G),
\end{equation*}
where $\alpha\in[0,1]$ is an arbitrary real number. It is clear that $A_{0}(G)=A(G)$ is the adjacency matrix of $G$, and $2A_{\frac{1}{2}}(G)=D(G)+A(G)=Q(G)$ is the signless Laplacian matrix of $G$. The eigenvalues of $A_{\alpha}(G)$ are called the $A_{\alpha}$-eigenvalues of $G$, and the largest of them, denoted by $\rho_{\alpha}(G)$, is called the $A_{\alpha}$-spectral radius of $G$. By Perron-Frobenius Theorem, the $A_{\alpha}$-spectral radius $\rho_{\alpha}(G)$ of a connected graph $G$ is simple and has an unit positive eigenvector. For some latest interesting results of the $A_{\alpha}$-matrix, the reader is invited to read the following recent papers \cite{Lin1,Lin2,Lin3,Liu1,Niki2,Niki3} and references therein.

A matching $M$ of $G$ is a set of disjoint edges of $G$, and the maximum size of a matching in $G$, denoted by $\mu(G)$ in here, is called the matching number of $G$. If each vertex of $G$ is incident with an edge in $M$, we call $M$ perfect. Obviously, a graph containing a perfect matching has an even number of vertices and $\mu(G)=\frac{|V(G)|}{2}$. Studies on connections between eigenvalues and the matching number of a graph emerged in the 1990s, and there have been abundant publications up to now, see \cite{Brou,Chang,Feng,Hou}. Recently, O \cite{O21} provided a lower bound for the adjacency spectral radius which guarantees the existence of a perfect matching in a connected graph $G$. Later, Liu et al. \cite{Liuc} and Zhao et al. \cite{Zhao} considered this problem with respect to the signless Laplacian spectral radius and $A_{\alpha}$-spectral radius, respectively.


The matching number of a graph is closely related to its (edge-) connectivity. The (edge-) connectivity of $G$ is the minimum number of vertices (edges) of which the deletion induces a non-connected graph or a single vertex. If the (edge-) connectivity of $G$ is at least $t$, where $t\geq0$, we say that $G$ is $t$-connected ($t$-edge-connected). In 2010,  Cioab\v{a} and O \cite{Cioa1} gave a sharp upper bound for the third largest eigenvalue of a $t$-edge-connected regular graph $G$ to guarantee that $G$ has a perfect matching, generalizing their results in \cite{Cioa2}. Very Recently, Zhang \cite{Wzhang} has studied the extremal graphs with maximum spectral radius among all $t$-connected graphs on $n$ vertices with matching number at most $\frac{n-k}{2}$, which generalized the results in \cite{O21}.

Along this line, we intend to investigate the maximum $A_{\alpha}$-spectral radius of $t$-connected graphs with bounded matching number. For a vertex subset $S\subseteq V(G)$, the subgraphs induced by the vertex set $V(G)-S$ and  $S$ are denoted by $G-S$ and  $G[S]$, respectively. The complement of $G$ is denoted by $\overline{G}$. For any two graphs $G_1$ and $G_2$, the joint of $G_1$ and $G_2$, denoted by $G_1\vee G_2$, is a graph whose vertex set $V(G_1\vee G_2)=V(G_1)\cup V(G_2)$ and edge set $E(G_1\vee G_2)=E(G_1)\cup E(G_2)\cup \left\lbrace uv|~u\in V(G_1), v\in V(G_2) \right\rbrace $. The disjoint of $G_1$ and $G_2$, denoted by $G_1\cup G_2$, is a graph with vertex set $V(G_1\cup G_2)=V(G_1)\cup V(G_2)$ and edge set $E(G_1\cup G_2)=E(G_1)\cup E(G_2)$.

In this paper, for $\alpha\in[0,\frac{1}{2}]$, we determine the extremal graphs with the maximum $A_{\alpha}$-spectral radius among all $t$-connected graphs on $n$ vertices with matching number $\frac{n-k}{2}$ at most. This generalizes some results of O \cite{O21} and Zhang \cite{Wzhang}.
\section{Preliminaries}

First, we present some essential lemmas in matrix theory. For a matrix $B$, we denote the largest eigenvalue of $B$ by $\lambda(B)$.

\begin{lemma}[Rayleigh-Ritz theorem]\label{symm}
	If $B$ is an $n\times n$ real symmetric matrix, then $\lambda(B)=\max\limits_{\textbf{x}\in \mathbb{R}^n}\dfrac{\textbf{x}^tB\textbf{x}}{\textbf{x}^{t}\textbf{x}}$.
\end{lemma}

\begin{lemma}\label{larger}
	Let $G$ be a connected graph and let $H$ be a connected subgraph of $G$. For $\alpha\in[0,1]$, we have $\rho_{\alpha}(H)\leq\rho_{\alpha}(G)$ with equality holding if and only if $H\cong G$.
\end{lemma}
\begin{proof}
	Let $n_0=|V(H)|$ and let $\boldsymbol{x}_{H}=(x^{H}_{i})^{t}_{i=1,\dots,n_0}$ be an unit positive eigenvector of $H$ corresponding to $\rho_{\alpha}(H)$. Then we construct the square matrix
	\begin{equation*}
		B=\begin{bmatrix}
			A_{\alpha}(H) & \boldsymbol{0} \\
			\boldsymbol{0} & \boldsymbol{0}
		\end{bmatrix}_{n\times n},
	\end{equation*}
and the vector $\boldsymbol{x}=(x^{H}_{1},\dots,x^{H}_{n_0},\underbrace{0,\dots,0}_{n-n_0})^{t}$. Obviously, $\lambda(B)=\lambda(A_{\alpha}(H))=\rho_{\alpha}(H)$, and $\boldsymbol{x}$ is the eigenvector of $B$ corresponding to $\lambda(B)$. It can be checked that
\begin{align*}
	\rho_{\alpha}(G)-\rho_{\alpha}(H)&=	\lambda(A_{\alpha}(G))-\lambda(B)\\
	&\geq \boldsymbol{x}^{t}\left( A_{\alpha}(G)-B\right) \boldsymbol{x}\\
	&\geq 0
\end{align*}
with equality holding if and only if $H\cong G$. This completes the proof.
\end{proof}

Next, we introduce the concepts of equitable matrices and equitable partitions.
\begin{definition}
	\label{def1}\cite{aebrouw}
	Let $B$ be a symmetric real matrix of order $n$ whose rows and columns are indexed by $P=\left\lbrace1,2,\cdots,n \right\rbrace $, where $\left\lbrace P_1,P_2,\cdots,P_r \right\rbrace $ is a partition of $P$ with $n_i=|P_i|$ and $n=n_1+n_2+\cdots+n_r$. Let $B$ be a matrix with the partition $\left\lbrace P_1,P_2,\cdots,P_r \right\rbrace$, i.e.
	\begin{equation*}
		M=\begin{pmatrix}
			B_{1,1} & B_{1,2} & \cdots & B_{1,r}\\
			B_{2,1} & B_{2,2} & \cdots & B_{2,r}\\
			\vdots & \vdots & \ddots & \vdots\\
			B_{r,1} & B_{r,2} & \cdots & B_{r,r}
		\end{pmatrix}_{n\times n},
	\end{equation*}
	where the blocks $B_{i,j}$ denotes the submatrix of $B$ formed by rows in $P_i$ and the $P_j$ columns. Let $m_{i,j}$ denote the average row sum of $B_{i,j}$. Then the matrix $M=(m_{i,j})$ is called the quotient matrix of $B$ under the given partition. Particularly, if the row sum of each submatrix $B_{i,j}$ is constant, then the partition is called equitable.
\end{definition}

\begin{lemma}\cite{lhYou}\label{equit}
	Let $M$ be an equitable quotient matrix of $B$ as defined in Definition \ref{def1}. If $B$ is a nonnegative matrix, then $\lambda(M)=\lambda(B)$.
\end{lemma}

\begin{theorem}[Berge-Tutte Formula]\label{betu}
	Let $G$ be a graph with matching number $\mu(G)$. Then 
	\begin{equation*}
		\mu(G)=\frac{n-\max_{S\subseteq V(G)}(o(G-S)-|S|)}{2},
	\end{equation*}
where $S$ is a subset of $V(G)$, and $o(G-S)$ is the number of odd components of $G-S$.
\end{theorem}

\section{Main results}
 From Fact 1 in \cite{Wzhang}, one see that if $G$ is an $n$-vertex connected graph with connectivity $t$ and matching number $\mu(G)<\lfloor\frac{n}{2}\rfloor$, then $t\leq \mu(G)$. Moreover, for a subset $S$ of $V(G)$, we have $t\leq|S|\leq\mu(G)$. First, we prove the following lemma.

\begin{lemma}
	Let $n$, $t$, $s$, $k$ be four positive integers, where $2\leq k\leq n-2$, $1\leq t\leq s \leq \frac{n-k}{2}$, and $n\equiv k$ (mod $2$). For a $t$-connected graph $G$ on $n$ vertices with matching number $\mu(G)\leq\frac{n-k}{2}$, we have $\rho_{\alpha}(G)\leq\rho_{\alpha}(K_{s}\vee(K_{n+1-2s-k}\cup\overline{K_{s+k-1}}))$ for $\alpha\in[0,1)$. Equality holds if and only if $G\cong K_{s}\vee(K_{n+1-2s-k}\cup\overline{K_{s+k-1}})$.
\end{lemma}
\begin{proof}
	Assume that $\alpha\in[0,1)$ and let $G$ be a $t$-connected graph with the maximum $A_{\alpha}$-spectral radius and matching number $\mu(G)\leq\frac{n-k}{2}$.
	From Theorem \ref{betu}, we see that $\max_{S\subseteq V(G)}(o(G-S)-|S|)=n-2\mu(G)$. Suppose that $S$ is the subset of $V(G)$ with the maximum number of vertices, such that $o(G-S)-|S|=n-2\mu(G)$. It can be checked that all components of $G-S$ are odd components (otherwise, we can randomly remove one vertex from each even component of $G-S$ and add them to the set $S$ until all components of $G-S$ are odd components, and the equality $o(G-S)-|S|=n-2\mu(G)$ always holds in the process). 
	
	For convenience, we denote $s=|S|$, $q=o(G-S)$. Note that $o(G-S)-|S|=n-2\mu(G)\geq k$, implying $q\geq s+k$. Let $G_1, G_2, \dots, G_q$ be the odd components of $G-S$ with $n_1, n_2, \dots, n_q$ vertices, respectively. Without loss of generality, we assume that $n_1\geq n_2\geq\dots\geq n_q\geq1$. Then we present the following three claims.

	\textbf{Claim 1.} Let $G'\cong K_s\vee\left( K_{n_1}\cup K_{n_2}\cup\dots\cup K_{n_q}\right)$. Then $\mu(G')\leq\frac{n-k}{2}$ and $\rho_{\alpha}\left(G \right)\leq \rho_{\alpha}\left( G'\right) $ with equality holding if and only if $G\cong G'$.
	\begin{proof}[Proof of Claim 1]
		Since $o(G'-S)=o(G-S)\geq s+k$ and $n-2\mu(G')\geq o(G'-S)-|S|$, we get $\mu(G')\leq\frac{n-k}{2}$. By Lemma \ref{larger}, it's easy to check that $\rho_{\alpha}(G)\leq\rho_{\alpha}(G')$ and equality holds if and only if $G\cong G'$. This proves Claim 1.
	\end{proof}

	\textbf{Claim 2.} Let $G''\cong K_s\vee\left( K_{n'_1}\cup K_{n'_2}\cup\dots\cup K_{n'_{s+k}}\right) $, where $n'_1=n_1+\sum_{i=s+k+1}^{q}n_{i}$, and $n'_i = n_i$ for $i=2,\dots,s+k$. Then $\mu(G'')\leq\frac{n-k}{2}$, and $\rho_{\alpha}\left(G' \right)\leq \rho_{\alpha}\left( G''\right)$ with equality holding if and only if $G'\cong G''$.
	\begin{proof}[Proof of Claim 2]
		Since $n_i$ ($i=1,2,\dots,n_q$) are all odd and $s+\sum_{i=1}^{q}n_i=n$, we see that $q+s\equiv n\equiv k$ (mod $2$), $q-s-k$ is even, and $n'_1=n_1+\sum_{i=s+k+1}^{q}n_{i}$ is odd. Note that $o(G''-S)=o(G-S)-(q-s-k)= s+k$  and $n-2\mu(G'')\geq o(G''-S)-|S|$, we obtain $\mu(G'')\leq\frac{n-k}{2}$. By Lemma \ref{larger}, we have $\rho_{\alpha}\left(G' \right)\leq \rho_{\alpha}\left( G''\right) $ with equality holding if and only if $G'\cong G''$. This proves Claim 2.
	\end{proof}

	
	\textbf{Claim 3.} Let $G'''\cong K_{s}\vee\left( K_{n+1-2s-k}\cup\overline{K_{s+k-1}}\right) $. Then $\mu(G''')\leq\frac{n-k}{2}$ and $\rho_{\alpha}\left(G'' \right)\leq \rho_{\alpha}\left( G'''\right)$ with equality holding if and only if $G''\cong G'''$.
	\begin{proof}[Proof of Claim 3]
		Note that $o(G'''-S)=o(G''-S)=s+k$ and $n-2\mu(G''')\geq o(G'''-S)-|S|$. It's easy to see that $\mu(G''')\leq\frac{n-k}{2}$. Now, we prove  $\rho_{\alpha}\left(G'' \right)\leq \rho_{\alpha}\left( G'''\right)$ for $[0,1)$. It suffices to show that $\rho_{\alpha}\left( K_s\vee\left( K_{n'_1}\cup K_{n'_2}\cup\dots\cup K_{n'_{s+k}}\right)  \right)<\rho_{\alpha}( K_s\vee\left( K_{n'_1+2}\cup\dots\cup K_{n'_j-2}\cup\dots\cup K_{n'_{s+k}}\right)$ whenever $n_j\geq3$ for $j=2,\dots,s+k$. Without loss of generality, we take $j=s+k$. By the proof of Claim 1 of Theorem 3 in \cite{Zhao}, we see that $\rho_{\alpha}\left( K_s\vee\left( K_{n'_1}\cup K_{n'_2}\cup\dots\cup K_{n'_{s+k}}\right)  \right)<\rho_{\alpha}\left( K_s\vee\left( K_{n'_1+2}\cup K_{n'_2}\cup\dots\cup K_{n'_{s+k-2}}\right)  \right)$. This inequality is strict. Hence, $\rho_{\alpha}\left(G'' \right)\leq \rho_{\alpha}\left( G'''\right)$ with equality holding if and only if $G''\cong G'''$. This proves Claim 3.
	\end{proof}
	Given the above, we conclude that for a $t$-connected graph $G$ on $n$ vertices with matching number $\mu(G)\leq\frac{n-k}{2}$, it follows that $\rho_{\alpha}(G)\leq\rho_{\alpha}\left( K_{s}\vee(K_{n+1-2s-k}\cup\overline{K_{s+k-1}})\right) $ for $\alpha\in[0,1)$, with equality holding if and only if $G\cong K_{s}\vee(K_{n+1-2s-k}\cup\overline{K_{s+k-1}})$. This completes the proof.
\end{proof}

	\begin{lemma}
		Let $n$, $t$, $k$ be three positive integers, where $2\leq k\leq n-2$, $1\leq t\leq\frac{n-k-4}{2}$, and $n\equiv k$ (mod $2$). For $ t+1\leq s\leq \frac{n-k-2}{2} $ and $\alpha\in[0,\frac{1}{2}]$, it follows that $\rho_{\alpha}(K_{s}\vee(K_{n+1-2s-k}\cup\overline{K_{s+k-1}}))<\max\left\lbrace \rho_{\alpha}(K_{t}\vee(K_{n+1-2t-k}\cup\overline{K_{t+k-1}})),\rho_{\alpha}\left( K_{\frac{n-k}{2}}\vee \overline{K_{\frac{n+k}{2}}} \right) \right\rbrace $. 
		
		(Note that $K_{s}\vee(K_{n+1-2s-k}\cup\overline{K_{s+k-1}})\cong K_{\frac{n-k}{2}}\vee \overline{K_{\frac{n+k}{2}}} $ for $s=\frac{n-k}{2}$).
	\end{lemma}
\begin{proof}
	Assume that $\alpha\in[0,\frac{1}{2}]$ and $t\leq s\leq \frac{n-k}{2}$. For convenience, we denote $\rho_{\alpha}(s)=\rho_{\alpha}( K_{s}\vee(K_{n+1-2s-k}\cup\overline{K_{s+k-1}})) $. Then $\rho_{\alpha}(t)=\rho_{\alpha}\left( K_{t}\vee(K_{n+1-2t-k}\cup\overline{K_{t+k-1}})\right) $ and $\rho_{\alpha}(\frac{n-k}{2})=\rho_{\alpha}\left(K_{\frac{n-k}{2}}\vee \overline{K_{\frac{n+k}{2}}} \right) $. The equitable quotient matrix of $A_{\alpha}\left( K_{s}\vee(K_{n+1-2s-k}\cup\overline{K_{s+k-1}})\right) $ with partition $\left\{ V(K_{s}), V(K_{n+1-2s-k}), \right.\\ \left. V(\overline{K_{s+k-1}}) \right\} $ has the form
	\begin{equation*}
		\begin{bmatrix}
			\alpha(n-1)+(1-\alpha)(s-1) & (1-\alpha)(n-2s-k+1)             & (1-\alpha)(s+k-1) \\
			(1-\alpha)s                 & \alpha(n-s-k)+(1-\alpha)(n-2s-k) & 0                 \\
			(1-\alpha)s                 & 0                                & \alpha s
		\end{bmatrix}.
	\end{equation*}

	By a simple calculation, we obtain the characteristic polynomial is
	\begin{align*}
		f_{1}(x,n,s,k,\alpha)&=x^3+[k-n+s+1-\alpha(n+s)]x^2+[ns\alpha^2+(n^2-kn-2s)\alpha+k-n+2s-ks-s^2]x\\
		&+(-k^2s+2kns-4ks^2+ks-n^2s+2ns^2-ns-3s^3+3s^2)\alpha^2+(3ns-3ks+7ks^2\\
		&+2k^2s-2ns^2-6s^2+5s^3-2kns)\alpha+ks-ns-3ks^2-k^2s+ns^2+2s^2-2s^3+kns.
	\end{align*}

	Referring to Lemma \ref{equit}, we see that $f_{1}(\rho_{\alpha}(s),n,s,k,\alpha)=0$. Note that $K_{n-s-k+1}$ is a subgraph of $K_{s}\vee(K_{n+1-2s-k}\cup\overline{K_{s+k-1}})$, then by Lemma \ref{larger}, we have $\rho_{\alpha}(s)>\rho_{\alpha}(K_{n-s-k+1})=n-s-k$. 
	
	The equitable quotient matrix of $A_{\alpha}\left( K_{\frac{n-k}{2}}\vee \overline{K_{\frac{n+k}{2}}} \right) $ with partition $\left\lbrace V\left(K_{\frac{n-k}{2}}, V(K_\frac{n+k}{2}) \right)  \right\rbrace $ equals
	\begin{equation*}
		\begin{bmatrix}
			\alpha(n-1)+(1-\alpha)\left( \frac{n+k}{2}-1 \right) & (1-\alpha)\frac{n+k}{2} \\
			(1-\alpha)\frac{n-k}{2} & \alpha\frac{n-k}{2}
		\end{bmatrix},
	\end{equation*}
and the corresponding characteristic polynomial is
\begin{equation*}
	f_2(x,n,k,\alpha)=x^2-\left( \frac{n-k}{2}-1+\alpha n \right)x+\left( -\frac{k^2}{4}-\frac{kn}{2}+\frac{k}{2}+\frac{3n^2}{4}-\frac{n}{2}\right)\alpha+\frac{k^2-n^2}{4}.
\end{equation*}

From Lemma \ref{equit}, we have $f_2\left(\rho_{\alpha}\left( \frac{n-k}{2}\right),n,k,\alpha \right)=0 $ and 
\begin{align*}
	\rho_{\alpha}\left(\frac{n-k}{2}\right)&=\frac{n-k-2+2\alpha n}{4}+\frac{\sqrt{(n-k-2+2\alpha n)^2+4(n^2-k^2)-4\alpha(3n+k-2)(n-k)}}{4}\\
	&>\frac{n-k-2}{2}.
\end{align*}

By a calculation, we obtain
\begin{equation*}
	\begin{split}
		&f_{1}(x,n,s,k,\alpha)-\left(x-\frac{n-k-2(1-\alpha)s}{2} \right)f_2(x,n,k,\alpha)\\
		&=\frac{n-2s-k}{8}\cdot([4k+4s-4-2\alpha(n+k-2)]x- (12\alpha-8)(1-\alpha)s^2\\
		&-[(2n-10k+12)\alpha^2+(18k+2n-24)\alpha-8k+8]s\\
		&-(n-k)[(1-\alpha)(n+k)-2\alpha(n-1)]).
	\end{split}
\end{equation*}

Let 
\begin{equation*}
	\begin{split}
		h(x,n,s,k,\alpha)&=[4k+4s-4-2\alpha(n+k-2)]x- (12\alpha-8)(1-\alpha)s^2-[(2n-10k+12)\alpha^2\\
		&+(18k+2n-24)\alpha-8k+8]s-(n-k)[(1-\alpha)(n+k)-2\alpha(n-1)],
	\end{split}
\end{equation*}
then we consider the following two situations.

\textbf{Case 1.} $n\leq 3k-3$. Since $n\equiv k$ (mod $2$), we consider $n\leq 3k-4$. Observe the following two inequalities
\begin{align}
	&4k+4s-4-2\alpha(n+k-2)>0\label{inq1}\\
	&\Leftrightarrow n<\left(\frac{2}{\alpha}-1\right)k+\frac{2}{\alpha}s-2\left(\frac{1}{\alpha}-1\right).\nonumber
\end{align}
and
\begin{align}
		&\rho_{\alpha}\left( \frac{n-k}{2}\right)\geq n-k\label{inq2}\\
		&\Leftrightarrow\frac{n-k-2+2\alpha n}{4}+\frac{\sqrt{(n-k-2+2\alpha n)^2+4(n^2-k^2)-4\alpha(3n+k-2)(n-k)}}{4}\geq n-k\nonumber\\
		&\Leftrightarrow n\leq 3k-4+\alpha(n-k+2)\nonumber\\
		&\Leftrightarrow n\leq \frac{1}{1-\alpha}(3k-4)-\frac{\alpha}{1-\alpha}(k-2).\nonumber
\end{align}

Obviously, for $\alpha\in[0,\frac{1}{2}]$ and $n\leq 3k-4$, the inequalities \eqref{inq1} and \eqref{inq2} hold. If $x\geq \rho_{\alpha}\left( \frac{n-k}{2}\right)$, we have
\begin{align*}
	h(x,n,s,k,\alpha)&\geq h\left(\rho_{\alpha}\left( \frac{n-k}{2}\right),n,s,k,\alpha \right)\geq h(n-k,n,s,k,\alpha)\\
	&=[4k+4s-4-2\alpha(n+k-2)](n-k)- (12\alpha-8)(1-\alpha)s^2-[(2n-10k+12)\alpha^2\\
	&+(18k+2n-24)\alpha-8k+8]s-(n-k)[(1-\alpha)(n+k)-2\alpha(n-1)]\\
	&=4(3\alpha-2)(\alpha-1)s^2+[-2(n-5k+6)\alpha^2-2(n+9k-12)\alpha+4n+4k-8]s\\
	&+(n-k)[3k-4-n+\alpha(n-k+2)].
\end{align*}

Denote $\varphi_1(\alpha)=4(3\alpha-2)(\alpha-1)$ and $\varphi_2(\alpha)=-2(n-5k+6)\alpha^2-2(n+9k-12)\alpha+4n+4k-8$. For $\alpha\in[0,\frac{1}{2}]$ and $n\leq 3k-4$, we have $\varphi_1(\alpha)>0$ and $-2(n-5k+6)>0$. Notice that
\begin{equation*}
	\frac{n+9k-12}{-2(n-5k+6)}>\frac{1}{2}\Leftrightarrow n>3-2k,
\end{equation*}
one can see $\varphi_2(\alpha)\geq\varphi_2(\frac{1}{2})=\frac{5}{2}(n-k)+1>0$. Thus, for $\alpha\in[0,\frac{1}{2}]$, $n\leq 3k-4$, $t\leq s<\frac{n-k}{2}$, and
$x\geq\rho_{\alpha}\left(\frac{n-k}{2}\right)$, we get
\begin{equation*}
	\begin{split}
		&~~~~f_{1}(x,n,s,k,\alpha)-\left(x-\frac{n-k-2(1-\alpha)s}{2} \right)f_2(x,n,k,\alpha)\\
		&=\frac{n-2s-k}{8}h(x,n,s,k,\alpha)\\
		&\geq \frac{n-2s-k}{8}\left\{\varphi_1(\alpha)s^2+\varphi_2(\alpha)s+(n-k)[3k-4-n+\alpha(n-k+2)]\right\}\\
		&>0,
	\end{split}
\end{equation*}
which implies that $\rho_{\alpha}(s)<\rho_{\alpha}(\frac{n-k}{2})$.

\textbf{Case 2.} $n\geq 3k-2$. We distinguish two cases depending on $s$.

\textbf{Case 2.1.} $s\leq \frac{n+3-3k}{4}$ (or $n\geq 4s+3k-3$). Consider the polynomial
\begin{align*}
	f_{1}(x,n,s,k,\alpha)&=x^3+[k-n+s+1-\alpha(n+s)]x^2+[ns\alpha^2+(n^2-kn-2s)\alpha+k-n+2s-ks-s^2]x\\
	&+(-k^2s+2kns-4ks^2+ks-n^2s+2ns^2-ns-3s^3+3s^2)\alpha^2+(3ns-3ks+7ks^2\\
	&+2k^2s-2ns^2-6s^2+5s^3-2kns)\alpha+ks-ns-3ks^2-k^2s+ns^2+2s^2-2s^3+kns,
\end{align*}
in the domain $D=\left\lbrace (x,n,s,k,\alpha)|~x>n-s-k, 1\leq t\leq s\leq\frac{n+3-3k}{4}, 2\leq k\leq n-2, \alpha\in[0,\frac{1}{2}] \right\rbrace $. We calculate that
\begin{equation*}
	\begin{split}
		\frac{\partial f_1}{\partial x}(x,n,s,k,\alpha)&=3x^2+2[k-n+s+1-(n+s)\alpha]x-s^2+(n\alpha^2-2\alpha-k+2)s\\
		&+k-n+\alpha n^2-\alpha kn.
	\end{split}
\end{equation*}

Observe the following inequality
\begin{align}
	&-\frac{k-n+s+1-(n+s)\alpha}{3}< n-s-k\nonumber\\
	&\Leftrightarrow (2-\alpha)n-(2+\alpha)s-2k+1>0\nonumber\\
	&\Leftrightarrow n>\frac{(2+\alpha)}{2-\alpha}s+\frac{2}{2-\alpha}k-\frac{1}{2-\alpha}\label{inq3}.
\end{align}

Since $\alpha\in[0,\frac{1}{2}]$, $s\geq t\geq 1$, $k\geq2$, and $n\geq 4s+3k-3$, then the inequality \eqref{inq3} holds. Therefore, we get
\begin{align*}
	\frac{\partial f_1}{\partial x}(x,n,s,k,\alpha)&>\frac{\partial f_1}{\partial x}(n-s-k,n,s,k,\alpha)\\
	&=ns\alpha^2+[(k-n)n+2s(s+k-1)]\alpha\\
	&+[(n-s-k)-(s+k-1)]n+k(s+k-1).
\end{align*}
It is routine to show
\begin{equation}
	-\frac{(k-n)n+2s(s+k-1)}{2ns}>1\Leftrightarrow n(n-k-2s)>2s(s+k-1).\label{inq4}
\end{equation}

Since $s\geq t\geq 1$, $k\geq 2$, and $n\geq 4s+3k-3$, we see the inequality \eqref{inq4} holds. Then
\begin{equation*}
	\begin{split}
		\frac{\partial f_1}{\partial x}(n-s-k,n,s,k,\alpha)&>\frac{\partial f_1}{\partial x}\left( n-s-k,n,s,k,\frac{1}{2}\right) \\
		&=\frac{1}{4}[2n^2+(4-7s-6k)n+4k^2+8ks-4k+4s^2-4s].
	\end{split}
\end{equation*}

Notice that
\begin{align}
	&-\frac{4-7s-6k}{4}< 4s+3k-3\Leftrightarrow 9s+6k> 8 \label{inq5}.
\end{align}

For $s\geq t\geq 1$, and $k\geq2$, the inequality \eqref{inq5} holds. It follows that
\begin{equation*}
	\begin{split}
		\frac{\partial f_1}{\partial x}(n-s-k,n,s,k,\frac{1}{2})&>\frac{\partial f_1}{\partial x}\left( 3s+2k-3,4s+3k-3,s,k,\frac{1}{2}\right) \\
		&=\frac{1}{4}(4k^2+11ks-10k+8s^2-15s+6)\\
		&=\frac{1}{4}[4(k-2)k+2k(s-1)+9s(k-1)+(8s-6)s+6]\\
		&>0
	\end{split}
\end{equation*}
for $s\geq t\geq 1$ and $k\geq 2$. By the above discussion, we have $\frac{\partial f_1}{\partial x}>0$ in the domain $D$. Moreover, we calculate that
\begin{equation*}
	\begin{split}
		\frac{\partial f_1}{\partial s}(x,n,s,k,\alpha)&=(1-\alpha)x^2+(n\alpha^2-2\alpha-k-2s+2)x-3(3\alpha-2)(\alpha-1)s^2\\
		&+[(4n-8k+6)\alpha^2+(14k-4n-12)\alpha-6k+2n+4]s\\
		&+[(k-n)-(n-k)^2]\alpha^2+[3(n-k)+(n-k)^2+(k-n)(n+k)]\alpha\\
		&+(k-n)+k(n-k).\label{inq6}
	\end{split}
\end{equation*}

It is routine to show
\begin{align}
		&-\frac{n\alpha^2-2\alpha-k-2s+2}{2(1-\alpha)}\leq n-s-k\nonumber\\
		&\Leftrightarrow [1+(1-\alpha)^2]n\geq (4-2\alpha)s+(3-2\alpha)k+2(1+\alpha)\nonumber\\
		&\Leftrightarrow n\geq \frac{4-2\alpha}{1+(1-\alpha)^2}s+\frac{3-2\alpha}{1+(1-\alpha)^2}k+\frac{2(1+\alpha)}{1+(1-\alpha)^2}.\label{inq7}
\end{align}

Note that $\alpha\in[0,\frac{1}{2}]$, $s\geq t\geq 1$, $k\geq2$, and $n\geq 4s+3k-3$. We have the inequality \eqref{inq7} holds. Hence,
\begin{align*}
	\frac{\partial f_1}{\partial s}(x,n,s,k,\alpha)&>\frac{\partial f_1}{\partial s}(n-s-k,n,s,k,\alpha)\\
	&=[(n-k-7s)(s+k-1)+2s(n-s-\frac{1}{2})]\alpha^2\\
	&+[(1-k-n)(n-k)+2s(2s+k-n)+10s(s+k-1)]\alpha\\
	&+(n-k)(n-k+1)-s(2n+3s+k-2).
\end{align*}

Notice that $(n-k-7s)(s+k-1)+2s(n-s-\frac{1}{2})\geq 2k(s+k-\frac{5}{2})+3s(s+k-\frac{7}{3})+3>0$ for $s\geq t\geq 1$ and $k\geq2$. Then we consider the following inequality
\begin{align}
	&-\frac{(1-k-n)(n-k)+2s(2s+k-n)+10s(s+k-1)}{2[(n-k-7s)(s+k-1)+2s(n-s-\frac{1}{2})]}>1\nonumber\\
	&\Leftrightarrow (4s+k-n)(n-s-k)-s(n-k-2)-(n-k)<0.\label{inq8}
\end{align}

It's easy to check that the inequality \eqref{inq8} holds for $s\geq t\geq1$, $k\leq n-2$, and $n\geq 4s+3k-3$.
Hence,
\begin{equation*}
	\begin{split}
		\frac{\partial f_1}{\partial s}\left( n-s-k,n,s,k,\alpha\right) &>\frac{\partial f_1}{\partial s}\left( n-s-k,n,s,k,\frac{1}{2}\right)\\
		&=\frac{1}{4} [2n^2+(5-9s-7k)n+5k^2+12ks-5k+7s^2-6s]
	\end{split}.
\end{equation*}

Since
\begin{equation*}
	-\frac{5-9s-7k}{4}< 4s+3k-3\Leftrightarrow 5s+5k>7
\end{equation*}
holds for $s\geq t\geq1$, and $k\geq 2$. It follows that
\begin{equation*}
	\begin{split}
		\frac{\partial f_1}{\partial s}(n-s-k,n,s,k,\frac{1}{2})&>\frac{\partial f_1}{\partial s}\left( 3s+2k-3,4s+3k-3,s,k,\frac{1}{2}\right) \\
		&=\frac{1}{4}(2k^2+5ks-5k+3s^2-7s+3)\\
		&=\frac{1}{4}[2k(k-2)+k(s-1)+4s(k-2)+3s^2+s+3]\\
		&>0.
	\end{split}
\end{equation*}
By the above discussion, one can see $\frac{\partial f_1}{\partial s}>0$ in the domain $D$. Referring to implicit function theorem, we get 
\begin{equation*}
	\frac{\partial x(s)}{\partial s}=-\dfrac{\frac{\partial f_1(x,n,s,k,\alpha)}{\partial s}}{\frac{\partial f_1(x,n,s,k,\alpha)}{\partial x}}<0
\end{equation*}
in the domain $D$, which implies that $x=x(s)$ is a strictly decreasing function for $t\leq s\leq \frac{n+3-3k}{4}$. As a result, for any $t+1\leq s\leq \frac{n+3-3k}{4}$, we deduce that $\rho_{\alpha}(s)<\rho_{\alpha}(t)$.

\textbf{Case 2.2.} $s\geq\frac{n+4-3k}{4}$ (or $n\leq 4s+3k-4$). Recall the function
\begin{equation*}
	\begin{split}
		&f_{1}(x,n,s,k,\alpha)-\left(x-\frac{n-k-2(1-\alpha)s}{2} \right)f_2(x,n,k,\alpha)\\
		&=\frac{n-2s-k}{8}\cdot([4k+4s-4-2\alpha(n+k-2)]x- (12\alpha-8)(1-\alpha)s^2\\
		&-[(2n-10k+12)\alpha^2+(18k+2n-24)\alpha-8k+8]s\\
		&-(n-k)[(1-\alpha)(n+k)-2\alpha(n-1)]).
	\end{split}
\end{equation*}

Since $n\leq 4s+3k-4\leq\left(\frac{2}{\alpha}-1\right)k+\frac{2}{\alpha}s-2\left(\frac{1}{\alpha}-1\right)$ for any $\alpha\in[0,\frac{1}{2}]$, then $4k+4s-4-2\alpha(n+k-2)>0$.
For $x\geq\rho_{\alpha}\left( \frac{n-k}{2}\right) $, we have
\begin{equation*}
	\begin{split}
		h(x,n,s,k,\alpha)&\geq h\left( \rho_{\alpha}\left( \frac{n-k}{2}\right),n,s,k,\alpha\right) \\
		&=\sqrt{(n-k-2+2\alpha n)^2+4(n^2-k^2)-4\alpha(3n+k-2)(n-k)}\\
		&\times \left[ -\frac{\alpha}{2}+\frac{1}{4}(4\alpha+4k+4s-2\alpha k-4)\right]+4(3\alpha-2)(\alpha-1)s^2\\
		&+\left[ (1-2\alpha^2)n+(10\alpha^2-18\alpha+7)k-12\alpha^2+24\alpha-10\right]s\\
		&+(2n-kn-n^2)\alpha^2+\left( -\frac{k^2}{2}+2k+\frac{5}{2}n^2-2n-2 \right)\alpha\\
		&+kn-n-k-n^2+2.
	\end{split}
\end{equation*}

For convenience, we denote 
\begin{equation*}
	\begin{split}
		g_1(n,s,k,\alpha)&=4(3\alpha-2)(\alpha-1)s^2+[ (1-2\alpha^2)n+(10\alpha^2-18\alpha+7)k\\
		&-12\alpha^2+24\alpha-10]s+(2n-kn-n^2)\alpha^2+\left( -\frac{k^2}{2}+2k+\frac{5}{2}n^2-2n-2 \right)\alpha\\
		&+kn-n-k-n^2+2,
	\end{split}
\end{equation*}
and
\begin{equation*}
	g_2(n,k,\alpha)=\sqrt{(n-k-2+2\alpha n)^2+4(n^2-k^2)-4\alpha(3n+k-2)(n-k)}.
\end{equation*}

Note that $1-2\alpha^2\geq \frac{1}{2}>0$, $n\geq 3k-2$, and $k\geq2$, we calculate that $(1-2\alpha^2)n+(10\alpha^2-18\alpha+7)k-12\alpha^2+24\alpha-10\geq (4k-8)\alpha^2+(24-18k)\alpha+10k-12>2k-2>0$. Thus, we get
\begin{equation*}
	\begin{split}
		h\left( \rho_{\alpha}\left( \frac{n-k}{2}\right),n,s,k,\alpha\right) &=g_1(n,s,k,\alpha)+\left[ -\frac{\alpha}{2}+\frac{1}{4}(4\alpha+4k+4s-2\alpha k-4)\right]g_2(n,k,\alpha)\\
		&\geq g_1\left( n,\frac{n+4-3k}{4},\alpha\right)+\left[ -\frac{\alpha}{2}+\frac{1}{4}(4\alpha+4k+4\cdot\frac{n+4-3k}{4}-2\alpha k-4)\right]\\
		&\times g_2(n,k,\alpha)\\
		&=\frac{1}{4}\left\{ [-3(n+k)^2+4(n+k)+8n]\alpha^2+[6(n+k)^2-(n-k)(n+k)\right.\\
		&\left.-16(n+k)-8(n-1)]\alpha+6(k+n)-2(n+k)^2+(n+k)(n-k) \right\}\\
		&+\left[ \frac{1}{4}(1-2\alpha)(n+k)+\alpha\right]g_2(n,k,\alpha)\\
		&=\frac{n+k}{4}\left\{ [-3(n+k)+4]\alpha^2+[6(n+k)-(n-k)-16]\alpha\right.\\
		&\left.+6-2(n+k)+(n-k)+(1-2\alpha)g_2(n,k,\alpha) \right\}\\
		&+\alpha\left\{ g_2(n,k,\alpha)-[2(n-1)-2n\alpha]\right\}.
	\end{split}
\end{equation*}

Let $\psi_1(n,k,\alpha)=[-3(n+k)^2+4(n+k)]\alpha^2+[6(n+k)-(n-k)-16]\alpha+6-2(n+k)+(n-k)+(1-2\alpha)g_2(n,k,\alpha)$ and $\psi_2(n,k,\alpha)=g_2(n,k,\alpha)-[2(n-1)-2n\alpha]$. Then we shall show that $\psi_1(n,k,\alpha)>0$ and $\psi_2(n,k,\alpha)>0$ for $\alpha\in[0,\frac{1}{2}]$, $k\geq2$, and $n\geq 3k-2$. 

For the function $\psi_2(n,k,\alpha)$, we calculate that
\begin{equation*}
	\begin{split}
		\psi_2(n,k,\alpha)&=\sqrt{(n-k-2+2\alpha n)^2+4(n^2-k^2)-4\alpha(3n+k-2)(n-k)}-[2(n-1)-2n\alpha]\\
		&=\frac{(n+k)[n-3k+4+4\alpha(k-2)]}{\sqrt{(n-k-2+2\alpha n)^2+4(n^2-k^2)-4\alpha(3n+k-2)(n-k)}+[2(n-1)-2n\alpha]}\\
		&> 0
	\end{split}
\end{equation*}
for $\alpha\in[0,\frac{1}{2}]$, $k\geq2$, and $n\geq 3k-2$.

To prove that $\psi_1(n,k,\alpha)>0$, we distinguish three cases depending on $k$.

(i) $k\geq4$.
We let
\begin{equation*}
	\tau_1(n,k,\alpha)=[-3(n+k)+4]\alpha+\left[ \frac{9}{2}(n+k)-16\right] -2g_2(n,k,\alpha),
\end{equation*}
and
\begin{equation*}
	\tau_2(n,k,\alpha)=g_2(n,k,\alpha)-\left[ 2(n+k)-(n-k)-6-\frac{3}{2}(n+k)\alpha+(n-k)\alpha\right].
\end{equation*}
It's easy to see that $\psi_1(n,k,\alpha)=\alpha\tau_1(n,k,\alpha)+\tau_2(n,k,\alpha)$.
By a calculation, we have
\begin{equation*}
	\begin{split}
		\tau_1(n,k,\alpha)&=\frac{1}{[-3(n+k)+4]\alpha+\left[ \frac{9}{2}(n+k)-16\right] +2g_2(n,k,\alpha)}\left[ \left( -7\alpha^2+5\alpha+\frac{1}{4}\right)n^2\right.\\
		&\left.+\left(132\alpha+\frac{97}{2}k-70\alpha k+18\alpha^2k-24\alpha^2-128 \right)n+9\alpha^2k^2-24\alpha^2k+16\alpha^2\right.\\
		&\left.-43\alpha k^2+164\alpha k-128\alpha+\frac{129}{4}k^2-160k+240 \right].
	\end{split}
\end{equation*}

Denote $\tilde{\tau}_1(n,k,\alpha)=\left( -7\alpha^2+5\alpha+\frac{1}{4}\right)n^2
+\left(132\alpha+\frac{97}{2}k-70\alpha k+18\alpha^2k-24\alpha^2-128 \right)n+9\alpha^2k^2-24\alpha^2k+16\alpha^2-43\alpha k^2+164\alpha k-128\alpha+\frac{129}{4}k^2-160k+240$. 
Note that $-7\alpha^2+5\alpha+\frac{1}{4}\geq \frac{1}{4}>0$. We calculate that
\begin{align}
	&-\frac{132\alpha+\frac{97}{2}k-70\alpha\label{inq9} k+18\alpha^2k-24\alpha^2-128}{2(-7\alpha^2+5\alpha+\frac{1}{4})}< 3k-2\\
	&\Leftrightarrow (-24\alpha^2-40\alpha+50)k+4\alpha^2+112\alpha-129> 0. \nonumber
\end{align}

For $\alpha\in[0,\frac{1}{2}]$ and $k\geq4$, we see that $-24\alpha^2-40\alpha+50\geq24>0$, and $(-24\alpha^2-40\alpha+50)k+4\alpha^2+112\alpha-129\geq-92\alpha^2-48\alpha+71\geq24>0$. It's easy to check that the inequality \eqref{inq9} holds and
\begin{equation*}
	\begin{split}
		\tilde{\tau}_1(n,k,\alpha)&>\tilde{\tau}_1(3k-2,k,\alpha)\\
		&=(36-48k)\alpha^2+(-208k^2+640k-372)\alpha+180k^2-644k+497.
	\end{split}
\end{equation*}

Since $36-48k<0$ for $k\geq4$, we have $\tilde{\tau}_1(3k-2,k,0)=180k^2-644k+497\geq801>0$ and $\tilde{\tau}_1(3k-2,k,\frac{1}{2})=76k^2-336k+320\geq192>0$. Consequently, we see $\tau_1(n,k,\alpha)=\frac{1}{[-3(n+k)+4]\alpha+\left[ \frac{9}{2}(n+k)-16\right] +2g_2(n,k,\alpha)}\tilde{\tau}_1(n,k,\alpha)>0$. Moreover, we calculate that
\begin{equation*}
	\begin{split}
		\tau_2(n,k,\alpha)&=\frac{1}{4\left\{g_2(n,k,\alpha)+\left[ 2(n+k)-(n-k)-6-\frac{3}{2}(n+k)\alpha+(n-k)\alpha\right] \right\}}\\
		&\times\left\{ (15\alpha^2-28\alpha+16)n^2+(48\alpha k-32k-24\alpha-10\alpha^2k+32)n-25\alpha^2k^2+76\alpha k^2\right.\\
		&\left.-152\alpha k-48k^2+160k-128\right\}.
	\end{split}
\end{equation*}

Denote $\tilde{\tau}_2(n,k,\alpha)=(15\alpha^2-28\alpha+16)n^2+(48\alpha k-32k-24\alpha-10\alpha^2k+32)n-25\alpha^2k^2+76\alpha k^2-152\alpha k-48k^2+160k-128$. Note that $15\alpha^2-28\alpha+16\geq5.75>0$ for $\alpha\in[0,\frac{1}{2}]$. Observe the following inequality
\begin{align}
	&-\frac{48\alpha k-32k-24\alpha-10\alpha^2k+32}{2(15\alpha^2-28\alpha+16)}<3k-2\label{inq10}\\
	&\Leftrightarrow (80\alpha^2-120\alpha+64)k-60\alpha^2+88\alpha-32> 0.\nonumber
\end{align}

Since $80\alpha^2-120\alpha+64\geq24>0$ for $\alpha\in[0,\frac{1}{2}]$, we obtain $(80\alpha^2-120\alpha+64)k-60\alpha^2+88\alpha-32\geq260\alpha^2-392\alpha+224\geq93>0$. Then the inequality \eqref{inq10} holds. It follows that
\begin{equation*}
	\begin{split}
		\tilde{\tau}_2(n,k,\alpha)&>\tilde{\tau}_2(3k-2,k,\alpha)\\
		&=(80\alpha^2-32\alpha)k^2+(-160\alpha^2+16\alpha+128)k+60\alpha^2-64\alpha-128.
	\end{split}
\end{equation*}

Notice that $(80\alpha-32)\alpha\geq4>0$ and $-160\alpha^2+16\alpha+128\geq96>0$ for $\alpha\in[0,\frac{1}{2}]$. Accordingly, we get $\tilde{\tau}_2(3k-2,k,\alpha)>\tilde{\tau}_2(10,4,\alpha)=700\alpha^2-512\alpha+384\geq303>0$ and $\tau_2(n,k,\alpha)=\frac{1}{4\left\{g_2(n,k,\alpha)+\left[ 2(n+k)-(n-k)-6-\frac{3}{2}(n+k)\alpha+(n-k)\alpha\right] \right\}}\tilde{\tau}_2(n,k,\alpha)>0$. 

By the above discussion, we have
\begin{equation*}
	\psi_1(n,k,\alpha)=\alpha\tau_1(n,k,\alpha)+\tau_2(n,k,\alpha)>0
\end{equation*}
for $\alpha\in[0,\frac{1}{2}]$, $k\geq2$, $n\geq3k-2$.

(ii) $k=3$. In this case, we have $n\geq 3k-2=7$. Let
\begin{equation*}
		\tau_3(n,k,\alpha)=[-3(n+k)+4]\alpha+\left[ 5(n+k)-16\right] -2g_2(n,k,\alpha)
\end{equation*}
and
\begin{equation*}
	\tau_4(n,k,\alpha)=g_2(n,k,\alpha)-\left[ 2(n+k)-(n-k)-6-(n+k)\alpha+(n-k)\alpha\right].
\end{equation*}

Obviously, we see $\psi_1(n,k,\alpha)=\alpha\tau_3(n,k,\alpha)+\tau_4(n,k,\alpha)$. Analogously, we calculate that
\begin{equation*}
	\begin{split}
		\tau_3(n,3,\alpha)&=[(5n-1)-(3n+5)\alpha]-2\sqrt{(n+2\alpha n-5)^2+4n^2-4\alpha(3n+1)(n-1)-36}\\
		&=\frac{(-7\alpha^2+2\alpha+5)n^2+(30\alpha^2-92\alpha+30)n+25\alpha^2-38\alpha+45}{[(5n-1)-(3n+5)\alpha]+2\sqrt{(n+2\alpha n-5)^2+4n^2-4\alpha(3n+1)(n-1)-36}}.
	\end{split}
\end{equation*}

Denote $\tilde{\tau}_3(n,3,\alpha)=(-7\alpha^2+2\alpha+5)n^2+(30\alpha^2-92\alpha+30)n+25\alpha^2-38\alpha+45$. Since $-7\alpha^2+2\alpha+5\geq\frac{17}{4}>0$ for $\alpha\in[0,\frac{1}{2}]$, we see that the inequality
\begin{equation*}
	-\frac{30\alpha^2-92\alpha+30}{2(-7\alpha^2+2\alpha+5)}<7\Leftrightarrow -68\alpha^2-64\alpha+100>0
\end{equation*}
holds. It follows that $\tilde{\tau}_3(n,3,\alpha)>\tilde{\tau}_3(7,3,\alpha)=-108\alpha^2-584\alpha+500\geq 181>0$ for $\alpha\in[0,\frac{1}{2}]$. Hence, $\tau_3(n,3,\alpha)=\frac{\tilde{\tau}_3(n,3,\alpha)}{[(5n-1)-(3n+5)\alpha]+2\sqrt{(n+2\alpha n-5)^2+4n^2-4\alpha(3n+1)(n-1)-36}}>0$ for $\alpha\in[0,\frac{1}{2}]$, $k=3$, and $n\geq7$.

By a simple calculation, we obtain
\begin{equation*}
	\begin{split}
		\tau_4(n,3,\alpha)&=\sqrt{(n+2\alpha n-5)^2+4n^2-4\alpha(3n+1)(n-1)-36}-[n+\alpha(n-3)-\alpha(n+3)+3]\\
		&=\frac{(4n^2-36)\alpha^2+(-8n^2+24n+48)\alpha+4n^2-16n-20}{\sqrt{(n+2\alpha n-5)^2+4n^2-4\alpha(3n+1)(n-1)-36}+[n+3-6\alpha]}.
	\end{split}
\end{equation*}

Denote $\tilde{\tau}_4(n,3,\alpha)=(4n^2-36)\alpha^2+(-8n^2+24n+48)\alpha+4n^2-16n-20$. Since $4n^2-36\geq 160>0$ for $n\geq7$, we have
\begin{equation*}
	\frac{8n^2-24n-48}{8n^2-72}>\frac{1}{2}\Leftrightarrow n^2-6n-3>0
\end{equation*}
holds. Notice that $n^2-6n-3\geq4>0$ for $n\geq7$, we obtain $\tilde{\tau}_4(n,3,\alpha)>\tilde{\tau}_4(n,3,\frac{1}{2})=n^2-4n-5\geq 16>0$ for $n\geq7$. Thus, one can see $\tau_4(n,3,\alpha)=\frac{\tilde{\tau}_4(n,3,\alpha)}{\sqrt{(n+2\alpha n-5)^2+4n^2-4\alpha(3n+1)(n-1)-36}+[n+3-6\alpha]}>0$ for $\alpha\in[0,\frac{1}{2}]$, $k=3$, and $n\geq7$.

By the above discussion, we conclude that
\begin{equation*}
	\psi_1(n,3,\alpha)=\alpha\tau_3(n,3,\alpha)+\tau_4(n,3,\alpha)>0.
\end{equation*}

(iii) $k=2$. In this case, one can see $n\geq 3k-2=4$. Let
\begin{equation*}
	\tau_5(n,k,\alpha)=[-3(n+k)+4]\alpha+\left[ 6(n+k)-(n-k)-16\right] -2g_2(n,k,\alpha)
\end{equation*}
and
\begin{equation*}
	\tau_6(n,k,\alpha)=g_2(n,k,\alpha)-\left[ 2(n+k)-(n-k)-6\right].
\end{equation*}
It's obvious that $\psi_1(n,k,\alpha)=\alpha\tau_5(n,k,\alpha)+\tau_6(n,k,\alpha)$. Then we calculate that 
\begin{equation*}
	\begin{split}
		\tau_5(n,2,\alpha)&=[(5n-2)-(3n+2)\alpha]-2\sqrt{(n+2\alpha n-4)^2+4n^2-12\alpha(n-2)n-16}\\
		&=\frac{(-7n^2+12n+4)\alpha^2+(2n^2-40n+8)\alpha+5n^2+12n+4}{[(5n-2)-(3n+2)\alpha]+2\sqrt{(n+2\alpha n-4)^2+4n^2-12\alpha(n-2)n-16}}.
	\end{split}
\end{equation*}

Denote $\tilde{\tau}_5(n,2,\alpha)=(-7n^2+12n+4)\alpha^2+(2n^2-40n+8)\alpha+5n^2+12n+4$. 
Since $-7n^2+12n+4\leq-60<0$ for $n\geq4$, we check that $\tilde{\tau}_5(n,2,0)=5n^2+12n+4\geq 132>0$ for $n\geq4$ and $\tilde{\tau}_5(n,2,\frac{1}{2})=\frac{1}{4}(17n^2-20n+36)\geq 57>0$ for $n\geq4$. Hence, $\tau_5(n,2,\alpha)=\frac{\tilde{\tau}_5(n,2,\alpha)}{[(5n-2)-(3n+2)\alpha]+2\sqrt{(n+2\alpha n-4)^2+4n^2-12\alpha(n-2)n-16}}>0$ for $\alpha\in[0,\frac{1}{2}]$, $k=2$, and $n\geq4$. Then we consider the function
\begin{equation*}
	\begin{split}
		\tau_6(n,2,\alpha)&=\sqrt{(n+2\alpha n-4)^2+4n^2-12\alpha n(n-2)-16}-n\\
		&=\frac{4(1-\alpha)n[(1-\alpha)n-2]}{\sqrt{(n+2\alpha n-4)^2+4n^2-12\alpha n(n-2)-16}+n}.
	\end{split}
\end{equation*}

It's easy to see that $4(1-\alpha)n[(1-\alpha)n-2]\geq0$ for any $\alpha\in[0,\frac{1}{2}]$ and $n\geq4$. Therefore,  we have $\tau_6(n,2,\alpha)\geq0$. By the above discussion, we deduce that
\begin{equation*}
	\psi_1(n,2,\alpha)=\alpha\tau_5(n,2,\alpha)+\tau_6(n,2,\alpha)>0,
\end{equation*}
for $\alpha\in[0,\frac{1}{2}]$, $k=2$, and $n\geq4$.

As can be seen, $\psi_1(n,k,\alpha)>0$ for $\alpha\in[0,\frac{1}{2}]$, $k\geq2$, and $n\geq3k-2$. Consequently,
%
for any $\alpha\in[0,\frac{1}{2}]$, $\frac{n+4-3k}{4}\leq s\leq \frac{n-k-2}{2}$, $k\geq2$, and $x\geq\rho_{\alpha}\left( \frac{n-k}{2}\right) $, we have
\begin{equation*}
	\begin{split}
		&~~~~f_{1}(x,n,s,k,\alpha)-\left(x-\frac{n-k-2(1-\alpha)s}{2} \right)f_2(x,n,k,\alpha)\\
		&\geq\frac{n-2s-k}{8}h\left( \rho_{\alpha}\left( \frac{n-k}{2}\right),n,s,k,\alpha\right) \\
		&\geq\frac{n-2s-k}{8} \left[ \frac{n+k}{4}\psi_1(n,k,\alpha)+\alpha\psi_2(n,k,\alpha)\right]\\
		&>0,
	\end{split}
\end{equation*}
which implies that $\rho_{\alpha}\left( s\right)<\rho_{\alpha}\left( \frac{n-k}{2}\right)$. 

This completes the proof.
\end{proof}

\begin{definition}\label{def2}
	Let
	\begin{equation*}
		\begin{split}
			\Delta_1(n,t,k,\alpha)&=\left(k+t-1-\frac{\alpha}{2}(n+k)+\alpha \right)\sqrt{(n-k-2+2\alpha n)^2+4(n^2-k^2)-4\alpha(3n+k-2)(n-k)}\\
			&+\left[ 12t^2+(10k-2n-12)t+2n-kn-n^2\right]\alpha^2+\left[ (24-18k)t-20t^2-\frac{k^2}{2}+2k+\frac{5}{2}n^2\right.\\
			&\left.-2n-2\right]\alpha+k^2-n^2+(k+t-1)(n+8t-k-2).
		\end{split}
	\end{equation*}
	and
	\begin{equation*}
		\begin{split}
			\Delta_2(n,t,k,\alpha)&=(\alpha-1)n^2+[4t\alpha^2+(2-14t-2k)\alpha+4k+12t-4]n-4t^2\alpha^2+(k^2+6kt-2k\\
			&+16t^2-4t)\alpha-3k^2-12kt+4k-16t^2+8t.
		\end{split}
	\end{equation*}
	where $n$, $t$, $k$ are three positive integers and $\alpha\in[0,\frac{1}{2}]$ is a real number. 
\end{definition}

\begin{theorem}\label{them1} 
	Let $n$, $t$, $k$ be three positive integers, where $2\leq k\leq n-2$, $1\leq t\leq\frac{n-k}{2}$ and $n\equiv k$ (mod $2$), and let $\alpha\in[0,\frac{1}{2}]$ be an arbitrary real number. The functions $\Delta_1(n,t,k,\alpha)$ and $\Delta_2(n,t,k,\alpha)$ are defined as in Definition \ref{def2}. For a $t$-connected graph $G$ on $n$ vertices with matching number $\mu(G)\leq\frac{n-k}{2}$, we obtain the following three conclusions.
	\begin{itemize}
		\item[\textnormal{(\romannumeral1)}] 
		If $\Delta_1(n,t,k,\alpha)<0$, then 
		\begin{equation*}
			\rho_{\alpha}(G)\leq\rho_{\alpha}\left( K_t\vee\left( K_{n+1-2t-k}\cup\overline{K_{t+k-1}} \right) \right)
		\end{equation*} 
		with equality holding if and only if $G\cong K_t\vee\left( K_{n+1-2t-k}\cup\overline{K_{t+k-1}} \right)$.
		\item[\textnormal{(\romannumeral2)}] If $\Delta_1(n,t,k,\alpha)=0$ and $\Delta_2(n,t,k,\alpha)\geq0$, then 
		\begin{equation*}
			\rho_{\alpha}(G)\leq\rho_{\alpha}\left( K_t\vee\left( K_{n+1-2t-k}\cup\overline{K_{t+k-1}} \right) \right)=\rho_{\alpha}\left( K_{\frac{n-k}{2}}\vee \overline{K_{\frac{n+k}{2}}}  \right)
		\end{equation*}
		 with equality holding if and only if $G\cong K_t\vee\left( K_{n+1-2t-k}\cup\overline{K_{t+k-1}} \right)$ or $G\cong K_{\frac{n-k}{2}}\vee \overline{K_{\frac{n+k}{2}}}$.
		\item[\textnormal{(\romannumeral3)}] If $\Delta_1(n,t,k,\alpha)>0$ and $\Delta_2(n,t,k,\alpha)\geq0$, then 
		\begin{equation*}
			\rho_{\alpha}(G)\leq\rho_{\alpha}\left( K_{\frac{n-k}{2}}\vee \overline{K_{\frac{n+k}{2}}}  \right)
		\end{equation*}
		with equality holding if and only if $G\cong K_{\frac{n-k}{2}}\vee \overline{K_{\frac{n+k}{2}}}$.
	\end{itemize} 
\end{theorem}
\begin{proof}
	Note that $K_t\vee\left( K_{n+1-2t-k}\cup\overline{K_{t+k-1}}\right)\cong K_{\frac{n-k}{2}}\vee \overline{K_{\frac{n+k}{2}}}$ if $t=\frac{n-k}{2}$, then we consider the case $t\leq \frac{n-k-2}{2}$, i.e., $n\geq 2t+k+2$. Similarly, it follows that
	\begin{equation*}
		\begin{split}
			&~~~~f_{1}(x,n,t,k,\alpha)-\left(x-\frac{n-k-2(1-\alpha)t}{2} \right)f_2(x,n,k,\alpha)\\
			&=\frac{n-2t-k}{8}\cdot([4k+4t-4-2\alpha(n+k-2)]x- (12\alpha-8)(1-\alpha)t^2\\
			&-[(2n-10k+12)\alpha^2+(18k+2n-24)\alpha-8k+8]t\\
			&-(n-k)[(1-\alpha)(n+k)-2\alpha(n-1)])\\
			&=\frac{n-2t-k}{8}h(x,n,t,k,\alpha).
		\end{split}
	\end{equation*}

Plugging $\rho_{\alpha}\left( \frac{n-k}{2}\right)$ into $h(x,n,t,k,\alpha)$ yields
	\begin{equation*}
		\begin{split}
			&~~~~h\left( \rho_{\alpha}\left( K_{\frac{n-k}{2}}\vee \overline{K_{\frac{n+k}{2}}}  \right),n,t,k,\alpha\right)\\
			&=\left(k+t-1-\frac{\alpha}{2}(n+k)+\alpha \right)\sqrt{(n-k-2+2\alpha n)^2+4(n^2-k^2)-4\alpha(3n+k-2)(n-k)}\\
			&+\left[ 12t^2+(10k-2n-12)t+2n-kn-n^2\right]\alpha^2+\left[ (24-18k)t-20t^2-\frac{k^2}{2}+2k+\frac{5}{2}n^2\right.\\
			&\left.-2n-2\right]\alpha+k^2-n^2+(k+t-1)(n+8t-k-2)\\
			&\triangleq\Delta_1(n,t,k,\alpha).
		\end{split}
	\end{equation*}

By Cauchy's Interlace Theorem \cite{Hwang}, one can see that the second largest root of $f_{1}(x,n,t,k,\alpha)$ is no bigger than $n-(2-\alpha)t-k$. By a calculation, we have
\begin{equation*}
	\begin{split}
		&\rho_{\alpha}\left( K_{\frac{n-k}{2}}\vee \overline{K_{\frac{n+k}{2}}}  \right)-[n-(2-\alpha)t-k]\geq0\\
		&\Leftrightarrow \sqrt{(n-k-2+2\alpha n)^2+4(n^2-k^2)-4\alpha(3n+k-2)(n-k)}\geq 4[n-(2-\alpha)t-k]-(n-k-2+2\alpha n)\\
		&\Leftrightarrow \Delta_2(n,t,k,\alpha)\geq 0.
	\end{split}
\end{equation*}

(i) If $\Delta_1(n,t,k,\alpha)<0$, we have
\begin{equation*}
	\begin{split}
		&~~~~f_{1}\left( \rho_{\alpha}\left( K_{\frac{n-k}{2}}\vee \overline{K_{\frac{n+k}{2}}}  \right),n,t,k,\alpha\right) -\left(\rho_{\alpha}\left( K_{\frac{n-k}{2}}\vee \overline{K_{\frac{n+k}{2}}}  \right)-\frac{n-k-2(1-\alpha)t}{2} \right)\\
		&\times f_2\left( \rho_{\alpha}\left( K_{\frac{n-k}{2}}\vee \overline{K_{\frac{n+k}{2}}}  \right),n,k,\alpha\right) \\
		&=\frac{n-2t-k}{8}h\left( \rho_{\alpha}\left( K_{\frac{n-k}{2}}\vee \overline{K_{\frac{n+k}{2}}}  \right),n,t,k,\alpha\right)\\
		&<0,
	\end{split}
\end{equation*}
implying $f_{1}\left( \rho_{\alpha}\left( K_{\frac{n-k}{2}}\vee \overline{K_{\frac{n+k}{2}}}  \right),n,t,k,\alpha\right)<0$. Hence, $\rho_{\alpha}\left( K_t\vee\left( K_{n+1-2t-k}\cup\overline{K_{t+k-1}} \right) \right)>\rho_{\alpha}\left( K_{\frac{n-k}{2}}\vee \overline{K_{\frac{n+k}{2}}}  \right)$.

(ii) If $\Delta_1(n,t,k,\alpha)=0$ and $\Delta_2(n,t,k,\alpha)\geq0$, we obtain
\begin{equation*}
	\begin{split}
		&~~~~f_{1}\left( \rho_{\alpha}\left( K_{\frac{n-k}{2}}\vee \overline{K_{\frac{n+k}{2}}}  \right),n,t,k,\alpha\right) -\left(\rho_{\alpha}\left( K_{\frac{n-k}{2}}\vee \overline{K_{\frac{n+k}{2}}}  \right)-\frac{n-k-2(1-\alpha)t}{2} \right)\\
		&\times f_2\left( \rho_{\alpha}\left( K_{\frac{n-k}{2}}\vee \overline{K_{\frac{n+k}{2}}}  \right),n,k,\alpha\right) \\
		&=\frac{n-2t-k}{8}h\left( \rho_{\alpha}\left( K_{\frac{n-k}{2}}\vee \overline{K_{\frac{n+k}{2}}}  \right),n,t,k,\alpha\right)\\
		&=0,
	\end{split}
\end{equation*}
which implies that $f_{1}(\rho_{\alpha}\left( K_{\frac{n-k}{2}}\vee \overline{K_{\frac{n+k}{2}}}  \right),n,t,k,\alpha)=0$ and $\rho_{\alpha}\left( K_{\frac{n-k}{2}}\vee \overline{K_{\frac{n+k}{2}}}  \right)\geq n-(2-\alpha)t-k$. Thus, we have $\rho_{\alpha}\left( K_t\vee\left( K_{n+1-2t-k}\cup\overline{K_{t+k-1}} \right) \right)=\rho_{\alpha}\left( K_{\frac{n-k}{2}}\vee \overline{K_{\frac{n+k}{2}}}  \right)$.

(iii) If $\Delta_1(n,t,k,\alpha)>0$ and $\Delta_2(n,t,k,\alpha)\geq0$, we obtain
\begin{equation*}
	\begin{split}
		&~~~~f_{1}(x,n,t,k,\alpha)-\left(x-\frac{n-k-2(1-\alpha)t}{2} \right)f_2(x,n,k,\alpha)\\
		&\geq f_{1}\left( \rho_{\alpha}\left( K_{\frac{n-k}{2}}\vee \overline{K_{\frac{n+k}{2}}}  \right),n,t,k,\alpha\right) -\left(\rho_{\alpha}\left( K_{\frac{n-k}{2}}\vee \overline{K_{\frac{n+k}{2}}}  \right)-\frac{n-k-2(1-\alpha)t}{2} \right)\\
		&\times f_2\left( \rho_{\alpha}\left( K_{\frac{n-k}{2}}\vee \overline{K_{\frac{n+k}{2}}}  \right),n,k,\alpha\right) \\
		&=\frac{n-2t-k}{8}h\left( \rho_{\alpha}\left( K_{\frac{n-k}{2}}\vee \overline{K_{\frac{n+k}{2}}}  \right),n,t,k,\alpha\right)\\
		&>0,
	\end{split}
\end{equation*}
which implies that $f_{1}(\rho_{\alpha}\left( K_{\frac{n-k}{2}}\vee \overline{K_{\frac{n+k}{2}}}  \right),n,t,k,\alpha)>0$ and $\rho_{\alpha}\left( K_{\frac{n-k}{2}}\vee \overline{K_{\frac{n+k}{2}}}  \right)\geq n-(2-\alpha)t-k$. We deduce that
$\rho_{\alpha}\left( K_t\vee\left( K_{n+1-2t-k}\cup\overline{K_{t+k-1}} \right) \right)<\rho_{\alpha}\left( K_{\frac{n-k}{2}}\vee \overline{K_{\frac{n+k}{2}}}  \right)$.

This completes the proof.
\end{proof}

Let $\alpha=0$, the Theorem 3.4 in \cite{Wzhang} can be proved easily through Theorem \ref{them1}. 

\begin{definition}\label{def3}
	Let
	$\Delta(n,t,k)=\left( n^2-k^2\right)^2-2\left(n^2-k^2 \right)(k+t-1)(n+k+10t-4)+16t(k+t-1)^2(n+4t-k-2) $, where $n$, $t$, $k$ are three integers.
\end{definition}
\begin{corollary}\cite{Wzhang}\label{coro1}
	Let $n$, $t$, $k$ be positive integers, where $2\leq k\leq n-2$, $1\leq t\leq\frac{n-k}{2}$ and $n\equiv k$ (mod $2$). The function $\Delta(n,t,k)$ is defined as in Definition \ref{def3}. Let $G$ be a $t$-connected graph on $n$ vertices with matching number $\mu(G)\leq\frac{n-k}{2}$. Then
	\begin{itemize}
		\item[\textnormal{(\romannumeral1)}] 
		if $\Delta(n,t,k)>0$, we get
		$\rho_{0}(G)\leq\rho_{0}\left( K_t\vee\left( K_{n+1-2t-k}\cup\overline{K_{t+k-1}} \right) \right)$ 
		with equality holding if and only if $G\cong K_t\vee\left( K_{n+1-2t-k}\cup\overline{K_{t+k-1}} \right)$.
		\item[\textnormal{(\romannumeral2)}] if $\Delta(n,t,k)=0$, we get
			$\rho_{0}(G)\leq\rho_{0}\left( K_t\vee\left( K_{n+1-2t-k}\cup\overline{K_{t+k-1}} \right) \right)=\rho_{0}\left( K_{\frac{n-k}{2}}\vee \overline{K_{\frac{n+k}{2}}}  \right)$
		with equality holding if and only if $G\cong K_t\vee\left( K_{n+1-2t-k}\cup\overline{K_{t+k-1}} \right)$ or $G\cong K_{\frac{n-k}{2}}\vee \overline{K_{\frac{n+k}{2}}}$.
		\item[\textnormal{(\romannumeral3)}] if $\Delta(n,t,k)<0$, we get 
			$\rho_{0}(G)\leq\rho_{0}\left( K_{\frac{n-k}{2}}\vee \overline{K_{\frac{n+k}{2}}}  \right)$
		with equality holding if and only if $G\cong K_{\frac{n-k}{2}}\vee \overline{K_{\frac{n+k}{2}}}$.
	\end{itemize} 
\end{corollary}
\begin{proof}
	By a calculation, we obtain
	\begin{equation*}
		\begin{split}
			\Delta_1(n,t,k,0)&=\left(k+t-1 \right)\sqrt{(n-k-2)^2+4(n^2-k^2)}-\left[n^2-k^2-(k+t-1)(n+8t-k-2) \right]\\
			&=\frac{-[\left( n^2-k^2\right)^2-2\left(n^2-k^2 \right)(k+t-1)(n+k+10t-4)+16t(k+t-1)^2(n+4t-k-2)]}{\left(k+t-1 \right)\sqrt{(n-k-2)^2+4(n^2-k^2)}+\left[n^2-k^2-(k+t-1)(n+8t-k-2) \right]}\\
			&=\frac{-\Delta(n,t,k)}{\left(k+t-1 \right)\sqrt{(n-k-2)^2+4(n^2-k^2)}+\left[n^2-k^2-(k+t-1)(n+8t-k-2) \right]}.
		\end{split}
	\end{equation*}

(i) If $\Delta(n,t,k)>0$, we have $\Delta_1(n,t,k,0)<0$. Then by Theorem \ref{them1} (i), it's easy to see that $\rho_{0}\left( K_{\frac{n-k}{2}}\vee \overline{K_{\frac{n+k}{2}}}  \right)<\rho_{0}\left( K_t\vee\left( K_{n+1-2t-k}\cup\overline{K_{t+k-1}} \right) \right)$.

(ii)-(iii)
Since $4k+4t-4-2\alpha(n+k-2)>0$ for $\alpha=0$, $t\geq1$ and $k\geq2$, we obtain that for any $x\geq \rho_{0}\left( K_{\frac{n-k}{2}}\vee \overline{K_{\frac{n+k}{2}}}  \right)$,
\begin{equation*}
	\begin{split}
				&~~~~f_{1}\left( x,n,t,k,0\right) -\left(x-\frac{n-k-2t}{2} \right)f_2\left( x,n,k,0\right)\\
				&\geq f_{1}\left( \rho_{0}\left( K_{\frac{n-k}{2}}\vee \overline{K_{\frac{n+k}{2}}}  \right),n,t,k,0\right) -\left(\rho_{0}\left( K_{\frac{n-k}{2}}\vee \overline{K_{\frac{n+k}{2}}}  \right)-\frac{n-k-2t}{2} \right)\\
				&\times f_2\left( \rho_{0}\left( K_{\frac{n-k}{2}}\vee \overline{K_{\frac{n+k}{2}}}  \right),n,k,0\right) \\
				&=\frac{n-2t-k}{8}h\left( \rho_{0}\left( K_{\frac{n-k}{2}}\vee \overline{K_{\frac{n+k}{2}}}  \right),n,t,k,0\right)\\
				&=\frac{-(n-2t-k)\Delta(n,t,k)}{8\left\lbrace \left(k+t-1 \right)\sqrt{(n-k-2)^2+4(n^2-k^2)}+\left[n^2-k^2-(k+t-1)(n+8t-k-2) \right]\right\rbrace }.
	\end{split}
\end{equation*}

If $\Delta(n,t,k)\leq0$, it can be checked that $f_1(x,n,t,k,0)\geq0$ for any $x\geq \rho_{0}\left( K_{\frac{n-k}{2}}\vee \overline{K_{\frac{n+k}{2}}}  \right)$, which implies that $\rho_{0}\left( K_t\vee\left( K_{n+1-2t-k}\cup\overline{K_{t+k-1}} \right) \right)\leq\rho_{0}\left( K_{\frac{n-k}{2}}\vee \overline{K_{\frac{n+k}{2}}}  \right)$. In this case, it's obvious that $\rho_{0}\left( K_{\frac{n-k}{2}}\vee \overline{K_{\frac{n+k}{2}}}  \right)>n-(2-\alpha)t-k$, implying $\Delta_2(n,t,k,0)\geq0$ holds. 
\end{proof}

Take $t=1$ and $k=2$, the following result can be obtained by Corollary \ref{coro1}.
\begin{corollary}\cite{O21}
	Let $n\geq4$ be an even integer, and let $G$ be a connected $n$-vertex graph. Then we obtain following two conclusion.
	\begin{itemize}
		\item[\textnormal{(\romannumeral1)}] 
		For $n\geq 8$, if
		$\rho_{0}(G)>\rho_{0}\left( K_1\vee\left( K_{n-3}\cup\overline{K_{2}} \right) \right)$, then $G$ contains a perfect matching.
		\item[\textnormal{(\romannumeral2)}] For $n=4,6$, if
		$\rho_{0}(G)>\rho_{0}\left( K_{\frac{n-2}{2}}\vee \overline{K_{\frac{n+2}{2}}}  \right)$, then $G$ contains a perfect matching.
	\end{itemize} 
\end{corollary}

By letting $t=1$, $k=2$ and $\alpha=\frac{1}{2}$ in Theorem \ref{them1}, the following corollary can be derived easily.

\begin{corollary}\cite{Liuc}
	Let $n\geq4$ be an even integer, and let $G$ be a connected $n$-vertex graph. Then we get following two conclusion.
		\begin{itemize}
		\item[\textnormal{(\romannumeral1)}] 
		For $n\geq 10$, if
		$\rho_{\frac{1}{2}}(G)>\rho_{\frac{1}{2}}\left( K_1\vee\left( K_{n-3}\cup\overline{K_{2}} \right) \right)$, then $G$ contains a perfect matching.
		\item[\textnormal{(\romannumeral2)}] For $n=4,6,8$, if
		$\rho_{\frac{1}{2}}(G)>\rho_{\frac{1}{2}}\left( K_{\frac{n-2}{2}}\vee \overline{K_{\frac{n+2}{2}}}  \right)$, then $G$ contains a perfect matching.
	\end{itemize} 
\end{corollary}

\section*{Declaration of competing interest}
The authors declare that they have no known competing financial interests or personal relationships that could have appeared to influence the work reported in this paper.

\section*{Acknowledgments}\setlength{\baselineskip}{15pt}
This work was supported by the National Natural Science Foundation of China [61773020]. The authors would like to express their sincere gratitude to the referees for their careful reading and insightful suggestions.

\end{document}